\documentclass[11pt]{amsart}
\usepackage{amsmath, amsfonts, amsthm, amssymb, multicol, mathtools, dsfont, verbatim}
\usepackage[shortlabels]{enumitem}
\usepackage[shortlabels]{enumitem}
\usepackage{graphicx}
\usepackage{tikz-cd}
\usepackage{lipsum}
\usepackage{mathtools}
\usepackage{float, hyperref}
\usepackage{scalerel,stackengine, subcaption}
\usepackage[shortlabels]{enumitem}
\usepackage{yfonts}
\usepackage[margin=1.2in]{geometry}
\usepackage{pgfplots}
\usepackage{tikz}
\usepgfplotslibrary{fillbetween}
\pgfplotsset{width=10cm,compat=1.9}
\usepackage[square,sort,comma,numbers]{natbib}
\allowdisplaybreaks

\usepackage[square,sort,comma,numbers]{natbib}
\DeclarePairedDelimiterX{\bignorm}[1]{\bigg\Vert}{\bigg\Vert}{#1}
\DeclarePairedDelimiterX{\norm}[1]{\lVert}{\rVert}{#1}

\usepackage{fancyhdr}

\makeatletter
\def\@setauthors{%
  \begingroup
  \def\thanks{\protect\thanks@warning}%
  \trivlist
  \centering\footnotesize \@topsep30\p@\relax
  \advance\@topsep by -\baselineskip
  \item\relax
  \author@andify\authors
  \def\\{\protect\linebreak}

  \normalsize\lowercase{\authors}%
  
	\ifx\@empty\contribs
  \else
    ,\penalty-3 \space \@setcontribs
    \@closetoccontribs
  \fi
  \endtrivlist
  \endgroup
}
\def\@settitle{\begin{center}
\LARGE\lowercase{\@title}
  \end{center}%
}
\makeatother

\definecolor{lightblue}{HTML}{2B77A4}
\definecolor{darkred}{HTML}{9E0D0D}
\hypersetup{
	colorlinks=true,
	linkcolor=darkred,
	urlcolor=darkred,
	citecolor=lightblue
}
\numberwithin{equation}{section}

\newtheorem{thm}{Theorem}[section]
\newtheorem{lma}[thm]{Lemma}

\newtheorem{defn}[thm]{Definition}

\newtheorem{prop}[thm]{Proposition}

\newtheorem{rem}[thm]{Remark}

\newtheorem*{thmA}{Theorem A}
\newtheorem*{thmB}{Theorem B}

\renewcommand{\epsilon}{\varepsilon}

\theoremstyle{definition}

\theoremstyle{conjecture}

\newcommand{\RR}{\mathbb{R}}

\newcommand{\ubox}{\overline{\textup{dim}}_\textup{B}}
\newcommand{\lbox}{\underline{\textup{dim}}_\textup{B}}

\newcommand{\haus}{{\textup{dim}}_\textup{H}}
\newcommand{\four}{{\textup{dim}}_\textup{F}}

\renewcommand{\geq}{\geqslant}
\renewcommand{\leq}{\leqslant}

\newcommand{\ubd}{\overline{\dim}_{\textup{B}}}
\newcommand{\lbd}{\underline{\dim}_{\textup{B}}}
\newcommand{\hd}{\dim_{\textup{H}}}

\newcommand{\fd}{\dim_{\mathrm{F}}}

\title{Fourier analytic variants of the Furstenberg\\ and Kakeya problems}

\author{ Jonathan M. Fraser${}^{1}$ and Lijian Yang${}^{2}$\\ \\
 ${}^{1,2}$ School of Mathematics and Statistics, University of St Andrews, Scotland\\
\MakeLowercase{Emails: ${}^{1}$jmf32@st-andrews.ac.uk and ${}^{2}$ly51@st-andrews.ac.uk}}
\thanks{JMF was  financially supported by a  \emph{Leverhulme Trust Research Project Grant} (RPG-2023-281)  and an \emph{EPSRC Open Fellowship} (EP/Z533440/1).}

\begin{document}

%\date{}

\maketitle
\thispagestyle{empty}

\begin{abstract}
We study several distinct but related Fourier analytic variants of the well-known Kakeya  and Furstenberg set problems in the plane.  For example, given $0<s,t<1$, we call a set $K \subseteq \mathbb{R}^2$ an $(s,t)$-Kakeya set if there exists a set of directions $E \subseteq S^1$ with Hausdorff dimension at least $t$ such that, for each $e \in E$, the set $K$ contains a subset of a unit line segment in direction $e$ whose Fourier dimension, viewed as a subset of $\mathbb{R}$, is at least $s$. For  $\Delta(s,t)$ defined to be  the infimum of the Fourier dimension among all $(s,t)$-Kakeya sets in $\mathbb{R}^2$, we prove that
\[
\Delta(s,t) = \min\{s,2t\}.
\]
We also obtain similar results in the Furstenberg set version of the problem and in the case where the Hausdorff dimension of the collection of lines is replaced by the Fourier dimension.\\



\end{abstract}

\tableofcontents

\section{Introduction and main results}

\subsection{Kakeya and Furstenberg problems}
Throughout this paper, the Kakeya and Furstenberg sets,
their associated direction sets, and the collections of lines
in the definitions below are assumed to be compact.
A set $K \subseteq \mathbb{R}^2$ is a \emph{Kakeya set}   if for every direction $e \in S^1$, there exists a unit line segment
\[
I_e = \{ a_e + t e : t \in [0,1] \}
\]
for some translation  $a_e \in \mathbb{R}^2$ such that $I_e \subseteq K$. That is, $K$ contains a unit line segment in every possible direction.  Kakeya-type problems lie at the intersection of geometric measure theory and harmonic analysis, and have deep connections with Fourier restriction phenomena and dimension theory.  There has been a huge amount of progress and attention on these problems recently, including the  spectacular resolution of the Kakeya conjecture in $\mathbb{R}^3$ by Wang and Zahl in 2025 \cite{wangzahl}.

In this paper we investigate the dimensions of various  Kakeya-type sets in $\mathbb{R}^2$ where  we restrict both the set of directions and the portion of the associated line segment contained in  the set.  More precisely, given $0<s,t\leq 1$, we consider Kakeya-type sets for which  there exists a set of directions $E \subseteq S^1$ with Hausdorff dimension at least $t$ such that, for each $e \in E$, the set $K$ contains a subset of a unit line segment in direction $e$ whose Fourier dimension, viewed as a subset of $\mathbb{R}$, is at least $s$. This $(s,t)$-framework therefore quantifies both the size of the direction set and the Fourier analytic structure of the Kakeya set within each line segment. The goal is then to understand how small the Fourier dimension of these Kakeya-type sets can be.  This analysis provides a refinement of work of Oberlin \cite{Oberlin} which established that the Fourier dimension of any Kakeya set in the plane is 2 (the maximal value possible).  We focus mainly on the Fourier dimension, but we also use the Hausdorff and box dimensions as well.  We recall the  definitions later in  Definition \ref{dimensionsdef}. The Fourier dimension is generally very difficult to determine or even to usefully  estimate.   For example, while the Hausdorff dimension of self-similar sets in $\RR$ is well understood under the open set condition, the Fourier dimension is known only in a few  special cases when it is either zero or the ambient spatial dimension, see \cite{amlan} and the references therein. For example, the middle-third Cantor set has Fourier dimension equal to zero and the unit interval has Fourier dimension 1.  However, for more general self-similar sets  the Fourier dimension remains poorly understood, and even the question of whether it is zero or strictly  positive is often open.

We now give the formal definition of the class of Kakeya-type sets we consider.  Since the subsets of the line segments are measured using \emph{Fourier} dimension and the direction set is measured using \emph{Hausdorff} dimension, we call this an FH-$(s,t)$-Kakeya set. That said, in order to simplify exposition, we use the term $(s,t)$-Kakeya set to mean an FH-$(s,t)$-Kakeya set in cases where the context is clear.

\begin{defn} \label{ksdef}
For $s,t \in (0,1]$, a set $K \subseteq \mathbb{R}^2$ is called an FH-$(s,t)$-Kakeya set if there exists a subset $E \subseteq S^1$ with $\haus E \geq t$ such that for each $e \in E$, there exist a point $a_e \in \mathbb{R}^2$ and a subset of parameters $R_e \subseteq [0,1]$ with $\four R_e \geq s$ such that
\begin{equation}
	L_e = \{ a_e + r e : r \in R_e \} \subseteq K. \nonumber
\end{equation}
\end{defn}

Our first main result, Theorem A in Section \ref{sec:kakeya}, obtains partial progress to understanding how small the Fourier dimension of FH-$(s,t)$-Kakeya sets can be.

The above considerations were partly inspired by the related \emph{Furstenberg set problem} and we consider Fourier analytic variants of this too.   Let $s \in (0,1]$ and $t \in (0,2]$. A set $K \subseteq \mathbb{R}^2$ is called an $(s,t)$-\emph{Furstenberg set} if there exists a collection of affine lines $\Lambda \subseteq A(2,1)$ with
\[
\haus \Lambda \geq t
\]
such that
\[
\haus(K \cap \ell) \geq s
\]
for every $\ell \in \Lambda$.   The Furstenberg set problem is then to determine how small the Hausdorff dimension of an $(s,t)$-Furstenberg set can be in terms of $s$ and $t$.  

 As with the Kakeya problem, the Furstenberg set problem is deeply connected to many important areas of mathematics, including harmonic analysis and incidence geometry.  Indeed, in a certain sense it can be considered a continuous analogue of the seminal Szemer\'edi--Trotter theorem.  Moreover, the problem also has seen significant attention and progress in recent years.    In particular, the works of Orponen and Shmerkin \cite{OS2021,OS2023} and Shmerkin and Wang \cite{SW2022} established a sequence of increasingly strong results and played an important role in the development of the subject prior to the recent complete solution of the Furstenberg problem by Ren and Wang. 
\begin{thm}[Ren--Wang {\cite[Theorem~1.1]{Furstenberg2023}}]
Let $K \subseteq \mathbb{R}^2$ be an $(s,t)$-Furstenberg set. Then
\begin{equation}
	\dim_{\textup{H}} K \geq \min\left\{s+t,\frac{3s+t}{2},s+1\right\}.\nonumber
\end{equation}
Moreover, this bound is sharp.
\end{thm}

One may also formulate box dimension and packing dimension variants of the  Furstenberg set problem by measuring both the collection of lines $\Lambda$ and the corresponding intersections $K \cap \ell$ using box dimension or packing dimension, respectively. More precisely, in the (upper) box dimension variant one assumes
\begin{equation} \label{boxvariant}
\ubox \Lambda \geq t
\qquad\text{and}\qquad
\ubox (K \cap \ell) \geq s \quad \text{for all } \ell \in \Lambda,
\end{equation}
while in the packing dimension variant one assumes
\begin{equation} \label{pvariant}
\dim_{\textup{P}} \Lambda \geq t
\qquad\text{and}\qquad
\dim_{\textup{P}} (K \cap \ell) \geq s \quad \text{for all } \ell \in \Lambda.
\end{equation}
In each case, the problem is to determine the smallest possible (upper) box dimension or packing dimension of a corresponding $(s,t)$-Furstenberg set $K$. These variants were resolved by Fraser \cite{boxdimensionff}, as follows.
\begin{thm}[Fraser {\cite[Theorems~1.1--1.4]{boxdimensionff}}]
Let $s \in (0,1]$ and $t \in (0,2]$ and consider the box dimension and packing dimension variants of the $(s,t)$-Furstenberg problem in the plane.
\begin{enumerate}
	\item For the (upper) box dimension version, if $K \subseteq \mathbb{R}^2$ is a bounded $(s,t)$-Furstenberg set as in \eqref{boxvariant}, then
	\begin{equation}
		\ubox K \geq \max\{s,t-1\}\nonumber
	\end{equation}
	and this bound is sharp.
	\item For the packing dimension version,  if $K \subseteq \mathbb{R}^2$ is an $(s,t)$-Furstenberg set as in \eqref{pvariant},  then
	\begin{equation}
		\dim_{\textup{P}} K \geq \max\left\{s,\frac{t}{2}\right\}\nonumber
	\end{equation}
	and this bound is sharp.
\end{enumerate}
\end{thm}

By considering different notions of dimension in place of Hausdorff dimension we gain a more holistic understanding of the underlying mechanisms which constrain the geometry of Kakeya-type and Furstenberg sets.  In this paper we consider the Fourier dimension version of the Furstenberg set problem and obtain partial progress, see Theorem B.  We also consider the variant where the size of the line set is measured by Hausdorff dimension.  We defer the precise formulations of these variants until Section \ref{sec:furstenberg}.

\subsection{Main results:~Kakeya variants} \label{sec:kakeya}

Given $s,t \in (0,1]$, let $\Delta_{\mathcal{K}}^{F,H}(s,t)$ be the infimal Fourier dimension of an FH-$(s,t)$-Kakeya set in $\mathbb{R}^2$ (defined above in Definition \ref{ksdef}), that is,
\[
\Delta_{\mathcal{K}}^{F,H}(s,t)
=
\inf \{ \four K : K \text{ is an FH-$(s,t)$-Kakeya set} \}.
\]
We now state our  main result which gives upper and lower bounds for $\Delta_{\mathcal{K}}^{F,H}(s,t)$.
\begin{thmA}\label{thma}
For all $s,t \in (0,1]$, 
\begin{equation}
  \Delta_{\mathcal{K}}^{F,H}(s,t)
	=
	\min\{s,2t\}. \nonumber
\end{equation}
\end{thmA}

\begin{rem}
The same results hold for $\Delta_{\mathcal{K}}^{F,F}(s,t)$, which is defined by replacing the Hausdorff dimension condition $\haus E \geq t$ on the direction set with  the stronger Fourier dimension condition $\four E \geq t$. The lower bound follows immediately from Theorem~A, and the same upper bound follows from Proposition~\ref{2texample} by taking the direction set $E$ to be a Salem set of dimension $t$.
\end{rem}We obtain  the lower bound in Theorem~A by proving that for every $(s,t)$-Kakeya set $K$,
\begin{equation}
	\four K \geq \min\{s,2t\},\nonumber
\end{equation}
see Theorem \ref{thm1} and its proof. The two upper bounds in Theorem~A come from rather different geometric constructions. The bound $s$ is obtained from a product-type example, while the bound $2t$ arises from a configuration in which all line segments share a common starting point. One example shows that there exists an $(s,1)$-Kakeya set $K_1$ with $\four K_1 = s$; see Proposition~\ref{sexample} for more details. The other  example is obtained by taking each $L_e$ to be the full unit interval while all segments share the same starting point; see Proposition~\ref{2texample} for more details. In fact, for the $(1,t)$-Kakeya set $K_0$ in Proposition~\ref{2texample}, we can compute that the Fourier dimension of $K_0$ is exactly $2t$. Oberlin \cite{Oberlin} proved that a $(1,t)$-Kakeya set containing full unit line segments has Fourier dimension at least $2t$, and our example $K_0$ shows that this bound is sharp. It is worth pointing out, however, that this does not immediately give an exact result for general $(1,t)$-Kakeya sets, even when all the line segments share a common starting point as in the construction of $K_0$. Indeed, a subset of a line with Fourier dimension $1$ need not contain any non-empty interval and there is a huge difference between these cases. Further, the result in \cite{restrictedkakeyaconjecture} shows that the Hausdorff dimension of such a restricted Kakeya set is at least $t+1$, which implies that $K_0$ is not a Salem set for $t<1$.

\subsection{Main results:~Furstenberg  variants} \label{sec:furstenberg}

 We first define two variants of the $(s,t)$-Furstenberg problem by imposing different dimensional conditions on the collection of lines. 
 \begin{defn}
 	Let $s \in (0,1]$ and $t \in (0,2]$. A set $E_{s,t}$ is called an FF-$(s,t)$-Furstenberg set if there exists a bounded collection of lines $\Lambda \subseteq A(2,1)$ such that
\begin{equation}
	\four (\Lambda) \geq t, \label{fourierconditiont}
\end{equation}
and for each $\ell \in \Lambda$,
\begin{equation}
	\four (\ell \cap E_{s,t}) \geq s,\nonumber
\end{equation}
where $\ell \cap E_{s,t}$ is viewed as a subset of $\mathbb{R}$. Similarly, we define $\widetilde{E}_{s,t}$ to be an FH-$(s,t)$-Furstenberg set by replacing the Fourier dimension condition \eqref{fourierconditiont} with the Hausdorff dimension condition
\begin{equation}
	\haus (\Lambda) \geq t. \label{hausdorffconditiont}
\end{equation}

 \end{defn}

 Throughout this paper, all collections of lines under consideration and the corresponding sets $E_{s,t}$ and $\widetilde{E}_{s,t}$ are assumed to be bounded.
We define the corresponding    Fourier dimension thresholds by
\[
\Delta^{F,F}_{\mathcal{F}}(s,t) = \inf \{ \four E_{s,t} : E_{s,t} \text{ is an FF-$(s,t)$-Furstenberg set} \}
\]
and
\[
\Delta^{F,H}_{\mathcal{F}}(s,t) = \inf \{ \four \widetilde{E}_{s,t} : \widetilde{E}_{s,t} \text{ is an FH-$(s,t)$-Furstenberg set} \}.
\]
As such one may consider the estimation of $\Delta^{F,F}_{\mathcal{F}}(s,t) $ to be the `pure Fourier' problem  in the sense that all the notions of dimension used are the Fourier dimension.  In this language, the original Furstenberg set problem is the `pure Hausdorff' problem  and the results of Fraser in \cite{boxdimensionff} resolve the `pure box' and `pure packing' problems. We resolve the `pure Fourier' problems, and our main results concerning Fourier analytic variants of the Furstenberg problem are as follows.

\begin{thmB} \label{thmb}
Let $s \in (0,1]$ and $t \in (0,2]$. Then, for the `pure Fourier' problem, 
\begin{equation}
\Delta^{F,F}_{\mathcal{F}}(s,t) = \min\{s,2t\}.\nonumber
\end{equation}
Further, in the `mixed problem' we have
\begin{equation}
	\begin{aligned}
		\Delta^{F,H}_{\mathcal{F}}(s,t) &= 0, && \text{if } 0 < t \leq 1,\\
		\min\{s,2(t-1)\} \leq \Delta^{F,H}_{\mathcal{F}}(s,t) &\leq s, && \text{if } 1 < t \leq 2.
	\end{aligned}\nonumber
\end{equation}
\end{thmB}

Thus, in contrast to the Hausdorff, box, and packing dimension variants, the Fourier dimension setting leads to a different family of thresholds. In particular, the FF and FH formulations exhibit different behaviour, reflecting the stronger instability and finer analytic nature of the Fourier dimension. We illustrate the lower and upper bounds for $\Delta^{F,H}_{\mathcal{F}}(s,t)$ in Figure~\ref{fig:asymptotic-furstenberg}. For fixed $s$, the blue and orange lines show the lower and upper bounds for $\Delta^{F,H}_{\mathcal{F}}(s,t)$, with the shaded region indicating the gap between them. The threshold equals $0$ for $0<t\leq1$ and $s$ for $1+s/2\leq  t\leq2$.

\begin{figure}[h]
	\centering
	\includegraphics[width=0.75\textwidth]{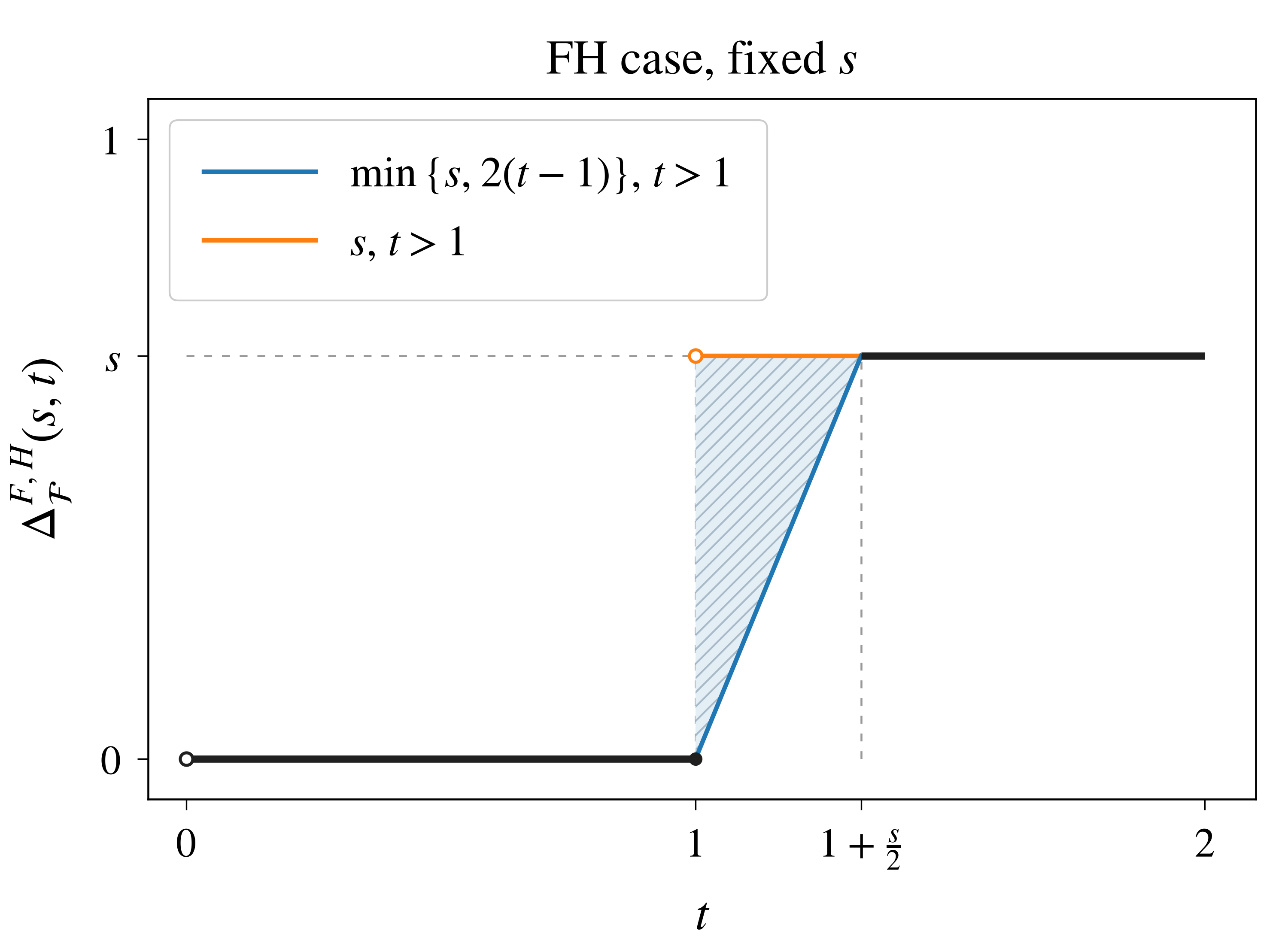}
	\caption{Lower and upper bounds for $\Delta^{F,H}_{\mathcal{F}}(s,t)$.}
	\label{fig:asymptotic-furstenberg}
\end{figure}
The lower bounds are proved in Proposition~\ref{furstenbergthm}, while the upper bounds are obtained by explicit constructions in the final part of the paper. In particular, for the FF case, we derive an upper bound from a condition on the lower box dimension of the collection of lines in Lemma~\ref{generalffupperbound}, and then show the existence of such an example in Proposition~\ref{existenceofambda}.

Naturally, we are interested in finding an exact formula for the remaining dimension threshold
\[
\Delta^{F,H}_{\mathcal{F}}(s,t).
\]
We leave this question for future work, but are not yet prepared to make any specific conjectures.

\subsection{Dimension theory, definitions, and some notation}

 Here and throughout we use the notation $\lesssim$ to denote inequalities up to an absolute constant. More precisely, for nonnegative quantities $A$ and $B$, we write $A \lesssim B$ if $A \leq C B$ for some constant $C>0$ independent of all relevant parameters. If the constant depends on a parameter $n$, we write $A \lesssim_n B$.  We now recall the definitions of the main notions of fractal dimension we use in this paper.  

\begin{defn} \label{dimensionsdef}
Let $X \subseteq \mathbb{R}^d$ be a bounded set.
\begin{enumerate}
	\item Let $\mathcal{M}(X)$ be the set of Borel probability measures supported on $X$. The \textit{Fourier dimension} of $X$ is defined by
	\begin{equation}
		\four X = \sup \left\{ 0 \leq s \leq d : \exists \mu \in \mathcal{M}(X)\text{ such that } |\widehat{\mu}(z)| \lesssim |z|^{-s/2} \right\}. \nonumber
	\end{equation}

	\item The \textit{Hausdorff dimension} of $X$ is defined by
	\begin{equation}
		\haus X
		=
		\inf\left\{
		\alpha \geq 0 :
		\forall \varepsilon > 0,\,
		\exists \{X_i\}_{i=1}^\infty
		\text{ such that }
		X \subseteq \bigcup_{i=1}^\infty X_i
		\text{ and }
		\sum_{i=1}^\infty \mathrm{diam}(X_i)^\alpha < \varepsilon
		\right\}. \nonumber
	\end{equation}

	\item For any $\delta > 0$, let $N_\delta(X)$ denote the smallest number of sets with diameter at most $\delta$ needed to cover $X$. The \textit{lower box-counting dimension} of $X$ is defined by
	\[
		\lbox X = \liminf_{\delta \to 0} \frac{\log N_\delta(X)}{-\log \delta}.
	\]

	\item The \textit{upper box-counting dimension} of $X$ is defined by
	\[
		\ubox X = \limsup_{\delta \to 0} \frac{\log N_\delta(X)}{-\log \delta}.
	\]
\end{enumerate}
\end{defn}The Fourier dimension enjoys only very limited stability properties. It is monotone with respect to set inclusion: if $A \subseteq B \subseteq \mathbb{R}^d$, then
\[
\four A \leq \four B.
\]
However, unlike the Hausdorff dimension, the Fourier dimension does not satisfy finite stability. In particular, it may happen that
\[
\four(A \cup B) > \max\{ \four A, \four B \},
\]
see \cite[Example~2]{badexample}. By contrast, the Hausdorff dimension is countably stable: for any countable collection of sets $\{A_j\}_{j=1}^{\infty} \subseteq \mathbb{R}^d$,
\[
\haus\!\left( \bigcup_{j=1}^{\infty} A_j \right)
= \sup_{j \geq 1} \haus A_j.
\]
As a result, the Fourier dimension behaves poorly under finite unions and must be treated with care in geometric constructions, such as those in this paper. 
The Fourier dimension is bounded above by the Hausdorff dimension. If $X \subseteq \mathbb{R}^d$ satisfies $\four X = \haus X$, then we call $X$ a Salem set. Fraser and Hambrook \cite{saleminRn} construct nontrivial Salem sets in $\RR^d$ with arbitrary dimensions. In particular, all Kakeya sets in $\mathbb{R}^2$ are Salem sets of Hausdorff and Fourier dimension $2$; see \cite{Oberlin}. More details and properties of the Fourier dimension can be found in \cite{M15}.

Next we fix some standard notation. We write  $\mathcal{S}(\mathbb{R}^d)$ to denote  the Schwartz space of rapidly decaying smooth functions on $\mathbb{R}^d$, and $B(x,r)$ denotes the closed Euclidean ball centered at $x$ with radius $r$.

We also recall the definition of the pushforward of a measure. Let $X$ and $Y$ be measurable spaces, and let $T : X \to Y$ be a measurable map. Given a Borel measure $\mu$ on $X$, the pushforward $T_{\#}\mu$ is the Borel measure on $Y$ defined by
\[
T_{\#}\mu(B) = \mu(T^{-1}(B))
\]
for every Borel set $B \subseteq Y$. Equivalently,
\[
\int_Y f(y)\, d(T_{\#}\mu)(y) = \int_X f(T(x))\, d\mu(x)
\]
for all bounded Borel measurable functions $f$ on $Y$.

 In the discussion of the Furstenberg problem, we always assume that the collection of lines $\Lambda \subseteq A(2,1)$ is bounded when viewed as a subset of $[0,\pi) \times \mathbb{R} \subseteq \mathbb{R}^2$. We parameterize $A(2,1)$ (and thus $\Lambda \subseteq A(2,1)$)    as follows. For each line $\ell \in A(2,1)$, let $e_{\ell}$ denote the direction of $\ell$ and let $a_{\ell}$ denote its translation parameter. More precisely, if $\ell$ corresponds to the pair $(\theta,a)$, then
\begin{equation} \label{linespara}
\ell = \{x \in \mathbb{R}^2 : x \cdot (-\sin\theta,\cos\theta) = a\},
\end{equation}
where $\theta \in [0,\pi)$ determines the direction $e_{\ell}$ and $a = a_{\ell}$ is the translation parameter.  In this way, we may discuss the Fourier dimension of $\Lambda$ by viewing it as a subset of the Euclidean plane.

\section{Proofs}

\subsection{Kakeya variants:~proof of Theorem~A}

For the lower bound in Theorem~A, we follow Oberlin's method \cite[Proposition~2]{Oberlin} to estimate the Fourier decay of a measure supported on $K$. In the proof of Theorem~\ref{thm1}, we first use the idea from Oberlin's paper to obtain the lower bound $\frac{2st}{s+2t}$. We then refine this approach to improve the lower bound to $\min\{s,2t\}$. In the argument below, we view each subset $L_e$ of a unit line segment in $\mathbb{R}^2$ as a subset of $\mathbb{R}$. Although the Fourier dimension of $L_e$ as a subset of $\mathbb{R}^2$ is zero, we do not distinguish between $L_e$ and $R_e$ in Definition~\ref{ksdef}, and thus write $\four L_e = \four R_e \geq s$.

\begin{thm}\label{thm1}
Let $K$ be an $(s,t)$-Kakeya set in $\mathbb{R}^2$. Then
\begin{equation}
	\min \{ s, 2t \} \leq \four K.\label{lowerb}
\end{equation}
\end{thm}

\begin{proof}
Let $E \subseteq S^1$ be the direction set of $K$. For each $e \in E$, let
\[
L_e = \{ a_e + r e : r \in R_e \}
\]
be the corresponding subset of a unit line segment contained in $K$, as in Definition~\ref{ksdef}. Since $K$ is an $(s,t)$-Kakeya set, we have $\haus E \geq t$ and $\four L_e \geq s$ for all $e \in E$.

Fix $\varepsilon > 0$. For each $e \in E$, there exists a measure $\nu_e$ supported on $R_e$ such that
\begin{equation}
	|\widehat{\nu_e}(\xi)| \leq c_e |\xi|^{-s/2+\varepsilon}\nonumber
\end{equation}
for all $\xi \in \mathbb{R}$, where $c_e$ depends only on $e$.

Define
\begin{equation}
	E_n = \left\{ e \in E : |\widehat{\nu_e}(\xi)| \leq n |\xi|^{-s/2+\varepsilon} \text{ for all } \xi \in \mathbb{R} \right\}.\nonumber
\end{equation}
Let $t_n = \haus E_n$, and we pick $n$ large enough so that $t_n-\varepsilon>0$. Then, by Frostman's lemma,  there exists a measure $\mu_n$ supported on $E_n$ such that
\begin{equation}
	\mu_n(B(x,r)) \lesssim_n r^{t_n - \varepsilon}.\nonumber
\end{equation}

Let $K_n \subseteq K$ be the union of $L_e$ for $e \in E_n$. Using $\nu_e$ and $\mu_n$, define a measure $\eta_n$ supported on $K_n$ by
\begin{equation} \label{measureeta}
	\int f(x) \, d \eta_n(x)
	= \int_{E_n} \int_{R_e} f(a_e + r e) \, d \nu_e(r) \, d \mu_n(e),
\end{equation}
and applying the Riesz representation theorem. This construction requires that  the map $e \mapsto a_e$ is $\mu_n$-measurable, which we assume without loss of generality. If it is not measurable, then a standard argument can be used to recover the proof. Briefly, decompose $E_n$ into finitely many sub-arcs $J_i \subseteq S^1$,
\[
E_n = \bigcup_{i=1}^N (J_i \cap E_n),
\]
and define
\[
\mu_{n,N} = \sum_{i=1}^N \mu_n(J_i)\delta_{e_i}
\]
for some $e_i \in J_i$. Let $\eta_{n,N}$ be defined analogously to \eqref{measureeta}. Then $\eta_n$ can be obtained as a weak-$*$ limit of $\{ \eta_{n,N} \}_{N \to \infty}$; see \cite[Proposition~2]{Oberlin} for more details.\\
We first prove the lower bound $\frac{2st}{s+2t}$.
The Fourier transform of the measure $\eta_n$ satisfies
\begin{equation}
\begin{aligned}
|\widehat{\eta_n}(\xi)|
&= \left| \int_{E_n} \int_{R_e} e^{-2\pi i (a_e + r e)\cdot \xi} \, d\nu_e(r) \, d\mu_n(e) \right| \\
&\leq \int_{E_n} \left| \int_{R_e} e^{-2\pi i r e \cdot \xi} \, d\nu_e(r) \right| d\mu_n(e) \\
&= \int_{E_n} |\widehat{\nu_e}(e \cdot \xi)| \, d\mu_n(e).
\end{aligned}\nonumber
\end{equation}
Define
\begin{equation}
	D_{\alpha} = \left\{ e \in E_n : \frac{|e \cdot \xi|}{|\xi|} \leq \alpha \right\}\nonumber
\end{equation}
and  split the above integral to obtain
\begin{equation}
|\widehat{\eta_n}(\xi)|
\leq \int_{D_{\alpha}} |\widehat{\nu_e}(e \cdot \xi)| \, d\mu_n(e)
 + \int_{E_n \setminus D_{\alpha}} |\widehat{\nu_e}(e \cdot \xi)| \, d\mu_n(e).\nonumber
\end{equation}
By the geometry of the unit circle, $D_{\alpha}$ is contained in two arcs of $S^1$, each of diameter comparable to $\alpha$, and hence
\begin{equation}
	\mu_n(D_{\alpha}) \lesssim_n \alpha^{t_n - \varepsilon}.\nonumber
\end{equation}
For $e \in E_n \setminus D_{\alpha}$, we have $|e \cdot \xi| > \alpha |\xi|$, and therefore
\begin{equation}
	|\widehat{\nu_e}(e \cdot \xi)|
	\lesssim_n |e \cdot \xi|^{-s/2+\varepsilon}
	\leq (\alpha |\xi|)^{-s/2+\varepsilon}.\nonumber
\end{equation}
Combining these estimates, we obtain
\begin{equation}
\begin{aligned}
|\widehat{\eta_n}(\xi)|
&\lesssim_n \alpha^{t_n-\varepsilon} + (\alpha |\xi|)^{-s/2+\varepsilon}.\nonumber
\end{aligned}
\end{equation}
Balancing the two terms by choosing \[
\alpha = |\xi|^{\frac{-s/2+\varepsilon}{t_n+s/2-2 \varepsilon}}
\]
gives
\begin{equation}
	|\widehat{\eta_n}(\xi)|
	\lesssim_n |\xi|^{\frac{(-s/2+\varepsilon)(t_n-\varepsilon)}{t_n+s/2-2 \varepsilon}}.\nonumber
\end{equation}
Therefore, the Fourier dimension of $K_n$ satisfies
\begin{equation}
	\four K_n \geq \frac{2(s/2-\varepsilon)(t_n-\varepsilon)}{t_n+s/2-2 \varepsilon}.\nonumber
\end{equation}
Since $\bigcup_n K_n = K$ and $t_n = \haus E_n$ satisfies $\lim_{n\to\infty} t_n = t$, letting $\varepsilon \to 0$ yields
\begin{equation}
	\four K \geq \frac{2st}{s+2t}. \nonumber
\end{equation}
We now prove the improved lower bound $\min\{s,2t\}$.
For $|\xi|\geq 2$, let $L=\lceil\log_2|\xi|\rceil$, where $\lceil\log_2|\xi|\rceil$ denotes the smallest integer greater than or equal to $\log_2|\xi|$. Define
\[
D_0=\left\{e\in E_n:
\frac{|e\cdot\xi|}{|\xi|}\leq\frac{1}{|\xi|}\right\}
\]
and, for $1\leq j\leq L$,
\[
D_j=\left\{e\in E_n:
\frac{2^{j-1}}{|\xi|}
<\frac{|e\cdot\xi|}{|\xi|}
\leq\frac{2^j}{|\xi|}\right\}.
\]
These sets partition $E_n$. Using the uniform bound
$\nu_e(R_e)\lesssim_{n,\varepsilon}1$ and the same arc estimate
as above, we obtain
\begin{equation}
\begin{aligned}
|\widehat{\eta_n}(\xi)|
&\leq \sum_{j=0}^{L}
\int_{D_j}|\widehat{\nu_e}(e\cdot\xi)|\,d\mu_n(e)\\
&\lesssim_{n,\varepsilon}
|\xi|^{-t_n+\varepsilon}
+\sum_{j=1}^{L}
2^{-j(s/2-\varepsilon)}
\left(\frac{2^j}{|\xi|}\right)^{t_n-\varepsilon}\\
&=|\xi|^{-t_n+\varepsilon}
\left(1+\sum_{j=1}^{L}2^{j(t_n-s/2)}\right).
\end{aligned}\nonumber
\end{equation}
The sum satisfies
\[
1+\sum_{j=1}^{L}2^{j(t_n-s/2)}
\leq
\begin{cases}
\displaystyle
\sum_{j=0}^{\infty}2^{j(t_n-s/2)}
=\frac{1}{1-2^{t_n-s/2}},
& t_n<s/2,\\[10pt]
\displaystyle
L+1\leq 2+\log_2|\xi|,
& t_n=s/2,\\[8pt]
\displaystyle
\frac{2^{(L+1)(t_n-s/2)}-1}{2^{t_n-s/2}-1}
\lesssim_n |\xi|^{t_n-s/2},
& t_n>s/2.
\end{cases}
\]
Consequently,
\[
|\widehat{\eta_n}(\xi)|
\lesssim_{n,\varepsilon}
\begin{cases}
|\xi|^{-t_n+\varepsilon}, & t_n<s/2,\\
(1+\log|\xi|)|\xi|^{-s/2+\varepsilon}, & t_n=s/2,\\
|\xi|^{-s/2+\varepsilon}, & t_n>s/2.
\end{cases}
\]
For every $\varepsilon>0$, there exists a constant
$C_\varepsilon$ such that
\[
1+\log|\xi|\leq C_\varepsilon|\xi|^\varepsilon
\qquad\text{for all }|\xi|\geq2.
\]
Thus, in the case $t_n=s/2$, we have
\[
(1+\log|\xi|)|\xi|^{-s/2+\varepsilon}
\leq C_\varepsilon|\xi|^{-s/2+2\varepsilon}.
\]
Combining the three cases, we conclude that
\[
|\widehat{\eta_n}(\xi)|
\lesssim_{n,\varepsilon}
|\xi|^{-\min\{s/2,t_n\}+2\varepsilon}.
\]
Therefore,
\[
\four K\geq\four K_n\geq\min\{s,2t_n\}-4\varepsilon.
\]
Since $E_n$ increases to $E$, we have
$\sup_n t_n=\haus E\geq t$. Letting $n\to\infty$ and then
$\varepsilon\to0$ gives
\[
\four K\geq\min\{s,2t\},
\]
as required.
\end{proof}
For the upper bound in Theorem~A, we construct two examples of $(s,t)$-Kakeya sets whose Fourier dimensions are bounded above by $s$ and $2t$, respectively. The first example is an $(s,1)$-Kakeya set.

\begin{prop}\label{sexample}
Let $X_1 \subseteq [0,1]$ be a compact set with $\four X_1 = s$. Then the set $K_1 = X_1 \times [0,1]$ is an $(s,1)$-Kakeya set in $\mathbb{R}^2$ with Fourier dimension
\begin{equation}
	\four K_1 = s.\nonumber
\end{equation}
\end{prop}

\begin{proof}
Since $K_1$ is the product of two compact subsets of $\mathbb{R}$, we may apply \cite[Theorem~3.2]{product} to get
\begin{equation}
	\four K_1 = \min\{s,\infty,2\} = s.\nonumber
\end{equation}
To see that $K_1$ is an $(s,1)$-Kakeya set, choose a family of line segments whose directions vary over a small interval corresponding to lines nearly parallel to the x-axis. This direction set has Fourier dimension $1$ (it is open in the parameter space), and for each such direction the intersection with $K_1$ is a copy of $X_1$, which has Fourier dimension $s$. Hence $K_1$ is an $(s,1)$-Kakeya set.
\end{proof}

The second example is constructed by taking all line segments to share the same starting point, namely the origin in $\mathbb{R}^2$, and to follow a direction set $E \subseteq S^1$ with $\haus E = t$.

\begin{prop}\label{2texample}
Given a direction set $E \subseteq S^1$ satisfying $\haus E = t$, define
\begin{equation}
	K_0 = \{ r e : r \in [0,1], \, e \in E \}.\nonumber
\end{equation}
%Assume further that the diameter of $E$ is less than $\frac{1}{10}$, which ensures that the angle between any two elements $e_1$ and $e_2$ is less than $\frac{\pi}{4}$.
 Then
\begin{equation}
	\four K_0 = 2t.\nonumber
\end{equation}
\end{prop}
\begin{rem}
The example in Proposition~\ref{2texample} shows that restricted Kakeya sets need not be Salem. Indeed, although the Fourier dimension of $K_0$ is exactly $2t$, its Hausdorff dimension is at least $t+1$ by \cite{restrictedkakeyaconjecture}. Therefore, for $0<t<1$, the Fourier and Hausdorff dimensions of $K_0$ do not coincide.
\end{rem}
In the proof of Proposition~\ref{2texample}, we will need the following lemma due to Mitsis \cite{projective}. For completeness, we include the proof below.

\begin{lma}\label{projective}
Suppose that $\alpha < d$ and $\mu$ is a finite Borel measure on $\mathbb{R}^d$ with Fourier decay
\begin{equation}
	|\widehat{\mu}(\xi)| \leq C_1 |\xi|^{-\alpha}\nonumber
\end{equation}
for all $\xi \in \mathbb{R}^d$. Then, for all $x \in \mathbb{R}^d$ and $r > 0$,
\begin{equation}
	\mu(B(x,r)) \leq C_1 C_2 r^{\alpha},\nonumber
\end{equation}
where the constant $C_2$ depends only on the Euclidean dimension $d$.
\end{lma}

\begin{proof}
Let $\phi \in \mathcal{S}(\mathbb{R}^d)$ be a non-negative function satisfying $\phi \geq 1$ on $B(0,1)$ and $\widehat{\phi}(\xi) = 0$ for $\xi \notin B(0,1)$. The existence of such a function is shown in, for example,  \cite[Example~3.2]{M15}. Define
\begin{equation}
	\phi_{x,r}(y) = \phi\!\left( \frac{x-y}{r} \right).\nonumber
\end{equation}
Then
\begin{equation}
\begin{aligned}
	\mu(B(x,r))
	&\leq \int \phi_{x,r}(y) \, d\mu(y) 
	\qquad  \\
	&= \int \widehat{\phi_{x,r}}(z)\, \overline{\widehat{\mu}(z)} \, dz 
	\qquad \text{(Fourier inversion and product formula)} \\
	&= r^d \int e^{-2\pi i\, x \cdot z}\, \widehat{\phi}(-r z)\, \overline{\widehat{\mu}(z)} \, dz 
	\qquad \text{(translation and dilation formula)} \\
	&\leq C_1 r^d \int_{|r z| \leq 1} |\widehat{\phi}(-r z)| \, |z|^{-\alpha} \, dz 
	\qquad \text{(Fourier decay for $\widehat{\mu}$)} \\
	&= C_1 r^{\alpha} \int_{|z| \leq 1} |\widehat{\phi}(z)| \, |z|^{-\alpha} \, dz 
	\qquad \text{(change of variables)} \\
	&= C_1 C_2 r^{\alpha}, \nonumber
\end{aligned}
\end{equation}
as required.
\end{proof}
\textbf{Proof of Proposition \ref{2texample}:}

One direction, namely $\four K_0 \geq 2t$, is due to Oberlin \cite{Oberlin}. For the convenience of the reader, we sketch the proof here, which is a simple modification of the proof of Theorem~\ref{thm1}. Since each line segment contained in $K_0$ has non-empty interior, we replace the measure $\nu_e$ in the proof of Theorem~\ref{thm1} by a Schwartz function $\varphi \in \mathcal{S}(\mathbb{R})$ supported on $L_e$. Then $\varphi$ has rapid Fourier decay, namely
\[
|\widehat{\varphi}(\xi)| \lesssim |\xi|^{-N}
\]
for all $N > 0$. Although $s$ denotes a Fourier dimension and is therefore at
most $1$, the argument in the proof of Theorem~\ref{thm1}
uses only the Fourier decay estimate. Since $\varphi$ has
arbitrarily rapid Fourier decay, we may repeat that argument
with arbitrarily large decay exponents, obtaining
$\four K_0\geq 2t$.

To prove $\four K_0 \leq 2t$, we argue by contradiction. Suppose that
\[
\four K_0 > 2t_1 > 2t_2 > 2t.
\]
Then there exists a measure $\mu_0$ supported on $K_0$ such that for all $\xi \in \mathbb{R}^2$,
\begin{equation}
	|\widehat{\mu_0}(\xi)| \lesssim |\xi|^{-t_1}.\nonumber
\end{equation}
For each $e \in E$, define the projected measure of $\mu_0$ in direction $e$ by
\begin{equation}
	\mu_{\perp_e} = (\pi_e)_{\#}\mu_0,\nonumber
\end{equation}
where the projection $\pi_e$ is given by
\[
\pi_e(x) = e + x - (x \cdot e)e.
\]
Let
\[
L_e^{\perp} = \{ x \in \mathbb{R}^2 : (x-e)\cdot e = 0 \}
\]
be the line perpendicular to $e$.

\begin{figure}[h]
	\centering
	\includegraphics[scale=1]{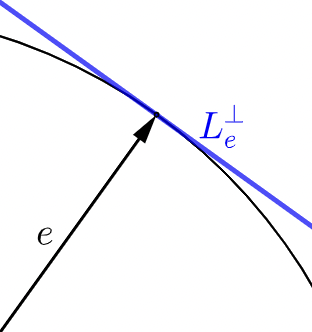}
	\caption{An illustration of $L_e^{\perp}$ which is the line tangent to $S^1$ and perpendicular to $e$.}
	\label{graphLe}
\end{figure}

Then $\mu_{\perp_e}$ is supported on $L_e^{\perp}$; see Figure~\ref{graphLe}. For $\xi \in L_e^{\perp}$, we compute
\begin{equation}
\begin{aligned}
	|\widehat{\mu_{\perp_e}}(\xi)|
	&= \left| \int_{L_e^\perp} e^{-2\pi i x \cdot \xi} \, d\mu_{\perp_e}(x) \right| \\
	&= \left| \int_{K_0} e^{-2\pi i (e + x - (x \cdot e)e)\cdot \xi} \, d\mu_0(x) \right| \\
	&= |\widehat{\mu_0}(\xi - e)|.\nonumber
\end{aligned}
\end{equation}
Since $e$ is a unit vector in $\mathbb{R}^2$, when $\xi \in L_e^{\perp}$ and $|\xi|$ is sufficiently large, we have
\begin{equation}
	|\widehat{\mu_{\perp_e}}(\xi)|
	= |\widehat{\mu_0}(\xi - e)|
	\lesssim |\xi|^{-t_1}.\nonumber
\end{equation}
It follows from Lemma~\ref{projective} that
\begin{equation}
	\mu_{\perp_e}(B(x,r)) \lesssim r^{t_1} \label{eq1}
\end{equation}
for all $x \in \mathbb{R}^2$ and $r > 0$, where the implicit constant is independent of $e$, $x$, and $r$.

On the other hand, let $\mu_E$ be the pushforward of $\mu_0$ onto the direction set $E$, that is,
\begin{equation}
	\mu_E = P_{\#}\mu_0,\nonumber
\end{equation}
where $P$ maps each line segment $L_e$ to its direction $e \in E$. Since $\haus E = t < t_2$, the measure $\mu_E$ does not satisfy the Frostman condition with exponent $t_2$. More precisely, there exist sequences $\{x_n\} \subset E$ and $\{r_n\} \subset \mathbb{R}$ with $r_n \to 0$ as $n \to \infty$ such that
\begin{equation}
	\mu_E(B(x_n,r_n)) > n r_n^{t_2}. \label{eq2}
\end{equation}
We claim that
\[
\mu_{\perp_{x_n}}(B(x_n,r_n)) \geq \mu_E(B(x_n,r_n)).
\]

\begin{figure}[h]
	\centering
	\includegraphics[scale=0.6]{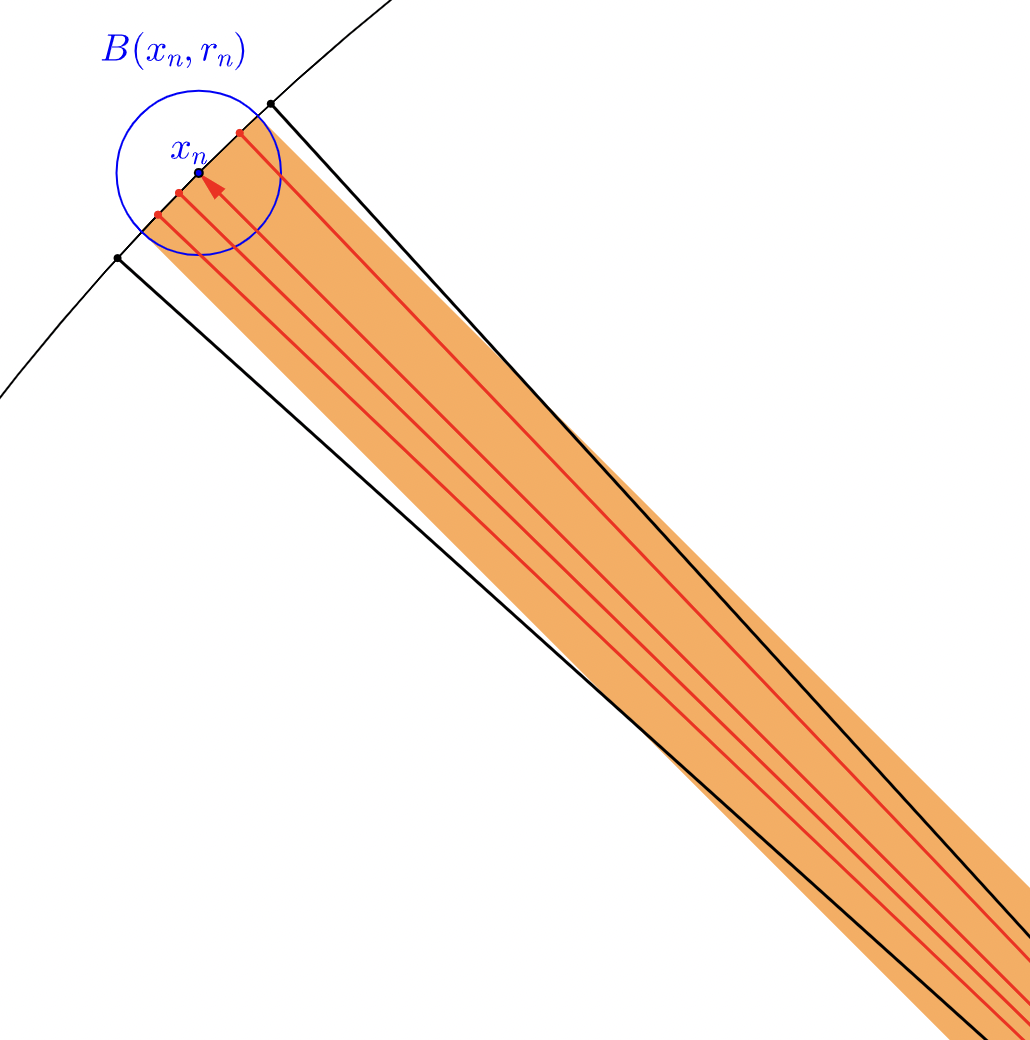}
	\caption{The red line segments are contained in $P^{-1}(B(x_n,r_n))$, while the red line segments together with parts of the black line segments are contained in the brown tube $\pi_{x_n}^{-1}(B(x_n,r_n))$}
	\label{twomeasures}
\end{figure}

Indeed, by definition of the measures $\mu_E$ and $\mu_{\perp_{x_n}}$, the preimage $P^{-1}(B(x_n,r_n))$ consists of all line segments with directions contained in $B(x_n,r_n)$. On the other hand, $\pi_{x_n}^{-1}(B(x_n,r_n))$ is a tube containing more than just those line segments; see Figure~\ref{twomeasures} for a geometric illustration.

Combining this claim with \eqref{eq1} and \eqref{eq2}, we obtain
\begin{equation}
	n r_n^{t_2}
	< \mu_E(B(x_n,r_n))
	\leq \mu_{\perp_{x_n}}(B(x_n,r_n))
	\lesssim r_n^{t_1}.\nonumber
\end{equation}
Letting $r_n \to 0$ yields a contradiction. This completes the proof. \hfill \qed

\subsection{Furstenberg variants:~proof of Theorem~B}

We now prove Theorem~B. The lower bounds are contained in the following proposition.

\begin{prop}\label{furstenbergthm}
Given $s \in (0,1]$, the following statements hold.

\begin{enumerate}
	\item If $t \in (0,2]$, then for any FF-$(s,t)$-Furstenberg set $E_{s,t}$,
	\begin{equation}
		\four E_{s,t} \geq \min \{ s, 2t \} .\nonumber
	\end{equation}

	\item If $t \in (1,2]$, then for any FH-$(s,t)$-Furstenberg set $\widetilde{E}_{s,t}$,
	\begin{equation}
		\four \widetilde{E}_{s,t} \geq \min \{ s, 2(t-1) \}.\nonumber
	\end{equation}
\end{enumerate}
\end{prop}

\begin{proof}
Let $P(\theta,a)=\theta$ be the projection onto the first
coordinate in the parametrisation \eqref{linespara}, and let
\[
D=\{(\cos\theta,\sin\theta):\theta\in P(\Lambda)\}
\]
be the corresponding direction set. Since the map
$\theta\mapsto(\cos\theta,\sin\theta)$ is bi-Lipschitz on
$[0,\pi)$, we have
\[
\haus D=\haus P(\Lambda).
\]

First suppose that $E_{s,t}$ is an FF-$(s,t)$-Furstenberg set.
For every $0<\varepsilon<t$, there exists a probability measure
$\eta$ supported on $\Lambda$ such that
\[
|\widehat{\eta}(\xi)|
\lesssim_\varepsilon
(1+|\xi|)^{-(t-\varepsilon)/2}.
\]
Let $\eta_P=P_{\#}\eta$. For every $u\in\mathbb{R}$,
\begin{equation}
\begin{aligned}
\widehat{\eta_P}(u)
&=\int_{\mathbb{R}}e^{-2\pi i u\theta}\,d\eta_P(\theta)\\
&=\int_{\Lambda}e^{-2\pi i uP(\theta,a)}\,d\eta(\theta,a)\\
&=\int_{\Lambda}e^{-2\pi i (u,0)\cdot(\theta,a)}
\,d\eta(\theta,a)\\
&=\widehat{\eta}(u,0).
\end{aligned}\nonumber
\end{equation}
Therefore, $\eta_P$ has the Fourier decay as
\[
|\widehat{\eta_P}(u)|
\lesssim_\varepsilon
(1+|u|)^{-(t-\varepsilon)/2}.
\]
Letting $\varepsilon\to0$, we obtain
\[
\haus D=\haus P(\Lambda)
\geq\four P(\Lambda)\geq\min\{t,1\}.
\]
For each $e\in D$, there exists a line $\ell\in\Lambda$
in direction $e$ such that
$\four(\ell\cap E_{s,t})\geq s$. Then $E_{s,t}$ contains an FH-$(s,t)$-Kakeya set.
Applying Theorem~\ref{thm1} gives
\[
\four E_{s,t}
\geq\min\{s,2\min\{t,1\}\}
=\min\{s,2t\}.
\]

Now suppose that $\widetilde{E}_{s,t}$ is an
FH-$(s,t)$-Furstenberg set with $t>1$.
Since $\Lambda$ is bounded, there exists $M>0$ such that
\[
\Lambda\subseteq P(\Lambda)\times[-M,M].
\]
By the product inequality
\cite[Product formula~7.3]{falconer},
\[
\haus\Lambda
\leq \haus\bigl(P(\Lambda)\times[-M,M]\bigr)
\leq \haus P(\Lambda)+\ubox[-M,M]
=\haus P(\Lambda)+1.
\]
Since $\haus\Lambda\geq t$, it follows that
\[
\haus D=\haus P(\Lambda)\geq t-1.
\]
For each $e\in D$, there exists a line $\ell\in\Lambda$
in direction $e$ with
$\four(\ell\cap\widetilde{E}_{s,t})\geq s$. Then $\widetilde{E}_{s,t}$ contains an FH-$(s,t-1)$-Kakeya set.
Applying Theorem~\ref{thm1} again gives
\[
\four\widetilde{E}_{s,t}\geq\min\{s,2(t-1)\},
\]
as required.
\end{proof}

For the upper bound $s$ in Theorem~B, it is easy to check that the example $K_1$ in Proposition~\ref{sexample} is both an FF-$(s,2)$-Furstenberg set and an FH-$(s,2)$-Furstenberg set, and that its Fourier dimension is equal to $s$.
Next we show the example for $\Delta^{F,H}_{\mathcal{F}}(s,t)=0$ when $t \in (0,1]$.  We choose a set $\Lambda \subseteq \mathbb{R}$ such that
\begin{equation}
	\four \Lambda = 0, \qquad \haus \Lambda = t,\nonumber
\end{equation}
and embed it into $[0,\pi) \times \mathbb{R}$ via $u \mapsto (0,u)$. Then the set $\widetilde{E}_{1,t} =  [0,1]\times \Lambda$ is an FH-$(1,t)$-Furstenberg set and, by  \cite[Theorem~3.2]{product}, the Fourier dimension of $\widetilde{E}_{1,t}$ is $0$.
Finally, for the upper bound $2t$ for $\Delta^{F,F}_{\mathcal{F}}(s,t)$ in Theorem~B, we need the following lemma, which connects the Fourier dimension with the lower box dimension of the collection of lines.

\begin{lma}\label{generalffupperbound}
Suppose that $E_{s,t}$ is a minimal FF-$(s,t)$-Furstenberg set associated to a collection of lines $\Lambda \subseteq [0,\pi)\times\mathbb{R}$ with $t<1$, in the sense that
\[
E_{s,t}=\bigcup_{\ell\in\Lambda} F_\ell,
\]
where $F_\ell\subseteq \ell$ and $\four F_\ell\geq s$ for each $\ell\in\Lambda$. If $\lbd \Lambda=t$, then
\[
\four E_{s,t}\leq 2t.
\]
\end{lma}

\begin{proof}
Let $N_{\delta}(\Lambda)$ be the smallest number of balls of radius $\delta$ needed to cover $\Lambda$. In what follows, for each $\ell \in \Lambda$, let $T_{\delta}^{e_{\ell}}(a_{\ell})$ denote a $\delta \times 1$ tube in $\mathbb{R}^2$ with direction $e_{\ell}$ and translation parameter $a_{\ell}$.

Note that in the parameter space $[0,\pi) \times \mathbb{R}$, a ball of radius $\delta$ consists of a collection of lines whose directions differ by at most $2\delta$ and whose translation parameters also differ by at most $2\delta$. Since $E_{s,t}$ is bounded, for any $\ell' \in \Lambda$, we may assume that for all $\ell \in B(\ell',\delta)$, the intersection $\ell \cap E_{s,t}$ is contained in a tube $T_{\delta}^{e'}(a')$.

For every $\varepsilon > 0$ satisfying $t+2\varepsilon <1$, there exists a sequence $\delta_n \to 0$ such that
\begin{equation}
	N_{\delta_n}(\Lambda) \lesssim \delta_n^{-(t+\varepsilon)}.\nonumber
\end{equation}
It follows that $E_{s,t}$ can be covered by $\delta_n^{-(t+\varepsilon)}$ many $T_{\delta_n}$-tubes.

We argue by contradiction. Suppose that
\begin{equation}
	\four E_{s,t} > 2(t+2\varepsilon).\nonumber
\end{equation}
Then there exists a probability measure $\mu$ supported on $E_{s,t}$ such that
\begin{equation}
	|\widehat{\mu}(\xi)| \lesssim |\xi|^{-(t+2\varepsilon)}.\nonumber
\end{equation}
By the pigeonhole principle, there exists a tube $T_{\delta_n}^{e_0}(a_0)$ such that
\begin{equation}
	\mu(T_{\delta_n}^{e_0}(a_0)) \gtrsim \delta_n^{t+\varepsilon}.\nonumber
\end{equation}
Project $\mu$ onto the line $L_{e_0}^{\perp}$. The projected measure $\mu_{\perp_{e_0}}$ then satisfies
\begin{equation}
	|\widehat{\mu_{\perp_{e_0}}}(z)| \lesssim |z|^{-(t+2\varepsilon)}\nonumber
\end{equation}
for all $z \in L_{e_0}^{\perp}$.
Let $b_0 \in L_{e_0}^{\perp}$ be the point obtained by projecting $a_0$ along the direction $e_0$ onto $L_{e_0}^{\perp}$. It follows from Lemma~\ref{projective} that
\begin{equation}
	\mu_{\perp_{e_0}}(B(b_0,\delta_n/2)) \lesssim \delta_n^{t+2\varepsilon}.\nonumber
\end{equation}
Since $T_{\delta_n}^{e_0}(a_0)$ projects into $B(b_0,\delta_n/2)$, we obtain
\begin{equation}
	\delta_n^{t+\varepsilon}
	\lesssim
	\mu(T_{\delta_n}^{e_0}(a_0))
	\leq
	\mu_{\perp_{e_0}}(B(b_0,\delta_n/2))
	\lesssim
	\delta_n^{t+2\varepsilon}.\nonumber
\end{equation}
Letting $n \to \infty$, we obtain a contradiction. Finally, letting $\varepsilon \to 0$, we deduce the desired upper bound $2t$.
\end{proof}

The final step is to show the existence of such a set $\Lambda$ satisfying
\begin{equation}
	\four \Lambda = \lbd \Lambda = t.\nonumber
\end{equation}

\begin{prop} \label{existenceofambda}
Let $t \in (0,1]$. Then there exists a compact set $E \subset \mathbb{R}^2$ such that
\[
\dim_F E = \dim_{\textup{H}} E = \lbd E = \ubox E =  t.
\]
\end{prop}

\begin{proof}
We first construct a compact set $A \subseteq [0,\infty)$ such that
\begin{equation}
	\haus A = \lbd A = \ubox A =  \frac{t}{2}. \label{lastexample}
\end{equation}
This can be achieved by taking $A$ to be an appropriate  self-similar set satisfying the open set condition, for which it is well-known that  the Hausdorff dimension and the box-counting dimension coincide. We refer the reader to \cite{falconer} for further details on self-similar sets and the open set condition.

Now consider the planar Wiener process $W : [0,\infty) \to \mathbb{R}^2$ and let $E = W(A)$. By Kahane's theorem \cite[Chapter 17, Theorem 1]{Kahane} on the image of planar Brownian motion, $E$ is almost surely a Salem set with Fourier dimension $t$. Moreover, since $W$ is almost surely $\alpha'$-H\"older continuous for every $\alpha' < \frac{1}{2}$, we obtain
\begin{equation}
	t= \fd E = \hd E \leq \lbd E \leq \ubd E \leq 2 \ubox A = t. \nonumber
\end{equation}
Here we used that the Hausdorff dimension of a bounded set is always bounded above by its lower box dimension.  This completes the proof.
\end{proof}	

\section*{AI Declaration}
The authors used ChatGPT-6 Astra to assist with discussions that contributed to the improved results in Theorem~\ref{thm1} and Proposition~\ref{furstenbergthm}.

\end{document}